\documentclass[12pt,french,american]{article}
\usepackage[margin=1in]{geometry}
\usepackage{mathrsfs, bbm, dsfont}
\usepackage{
amsmath,
amsfonts,
latexsym, 
amsrefs,
amssymb,
mathtools,
tikz,
tabularx,
minibox,
amsthm,
}
\usepackage{bm}
\usepackage{babel}
\usepackage{floatrow}
\usepackage[T1]{fontenc}
\usepackage{graphicx}
\DeclareGraphicsExtensions{.eps}
\usepackage{url}
\usepackage{wrapfig}
\usepackage{pdflscape}
\usepackage{afterpage}
\usepackage{tocloft}

\usepackage{caption}
\usepackage{subcaption}
\usepackage{enumerate}
\definecolor{mygray}{gray}{0.3} 
\usepackage[colorlinks=true,urlcolor=gray,linkcolor=mygray,citecolor=mygray]{hyperref}

\usetikzlibrary{shapes}
\usetikzlibrary{intersections}
\usetikzlibrary{svg.path}
\usetikzlibrary{decorations.markings}
\usetikzlibrary{hobby}

\numberwithin{equation}{section}
\renewcommand{\rm}{\mathrm}

\newcommand{\be}{\begin{equation}}
\newcommand{\ee}{\end{equation}}

\newcommand{\ga}{{\gamma}}
\newcommand{\Ga}{{\Gamma}}
\newcommand{\la}{\lambda}

\newcommand{\si}{\sigma}

\newcommand{\AdPair}{\mathscr{P}}

\newcommand{\disteq}{\stackrel{d}{=}}

\newcommand{\E}{\mathbb{E}}
\newcommand{\R}{\mathbb{R}}

\newcommand{\tr}{\mathrm{tr}}

\theoremstyle{plain} 
\newtheorem{theorem}{Theorem}
\newtheorem*{theorem*}{Theorem}
\newtheorem{lemma}[theorem]{Lemma}

\newtheorem*{lemma*}{Lemma}
\newtheorem{corollary}[theorem]{Corollary}
\newtheorem*{corollary*}{Corollary}
\newtheorem{proposition}[theorem]{Proposition}
\newtheorem*{proposition*}{Proposition}

\newtheorem{fact}[theorem]{Fact}
\newtheorem*{fact*}{Fact}

\newtheorem*{definition*}{Definition}

\newtheorem*{example*}{Example}
\newtheorem{remark}[theorem]{Remark}

\newtheorem*{remark*}{Remark}
\newtheorem*{remarks*}{Remarks}

\makeatletter
\renewcommand{\section}{\@startsection
{section}
{1}
{0mm}
{-2\baselineskip}
{1\baselineskip}
{\normalfont\large\scshape\centering}} 
\makeatother

\makeatletter
\renewcommand{\subsection}{\@startsection
{subsection}
{2}
{0mm}
{-\baselineskip}
{1 \baselineskip}
{\normalfont\scshape}} 
\makeatother

\makeatletter
\renewcommand{\subsubsection}{\@startsection{subsubsection}{3}{\z@}%
  {3.25ex \@plus 1ex \@minus .2ex}{-1em}{\normalfont\normalsize\itshape}}
\makeatother
\usepackage[T1]{fontenc}

\title{\Large \textbf{ \scshape On the number of cycles in commutators of random permutations}}

\author{Guillaume Dubach\footnote{gdubach@ist.ac.at} \\ 
IST Austria\footnote{Am Campus 1, 3400 {\scshape Klosterneuburg, Austria}}}
\date{}

\begin{document}
\maketitle

\vspace{-.2in}
\begin{abstract}
We present general links between statistics of non-Hermitian random matrices and the distribution of the number of cycles of some specific random permutations. In particular, we derive explicit formulas for the generating functions of the number of cycles in the commutator $[\sigma,\tau] = \sigma \tau \sigma^{-1} \tau^{-1}$ where $\sigma$ is uniformly distributed, and $\tau$ is either one cycle, the product of two cycles of same size, or the product of many transpositions.
\end{abstract}


\section{Random permutations}\label{sec:rp}

\subsection{General introduction}

The number of cycles $\mathscr{C}(\sigma)$ of a random permutation $\sigma$ is one the most elementary aspects one might want to describe; its distribution is well known and particularly elegant when $\sigma$ is distributed uniformly in $\mathfrak{S}_M$. This famous result will serve here as a paradigm.

\begin{theorem}[Cycles of a uniform permutation]\label{thm:uniform_permutation} For $M \geq 1$, if $\sigma$ is a uniform element of $\mathfrak{S}_M$, then
\begin{equation}\label{uniform_gen_fun_1}
\E \left( t^{\mathscr{C}(\si)} \right)
= \frac{1}{M!} t (t+1) \cdots (t+M-1),
\end{equation}
which is equivalent to the equality in distribution
\begin{equation}
\mathscr{C}(\si) \disteq B_1 + B_{2} + \cdots + B_{M}
\end{equation}
where $(B_k)_{k \geq 1}$ is a sequence of independent Bernoulli variables with $B_k \disteq \mathrm{Ber}\left( \frac{1}{k}\right)$.
\end{theorem}

Polynomials such as \eqref{uniform_gen_fun_1} appear repeatedly in this context. This motivates the following notation: for any integer $n$, we define the rising and falling factorials as the polynomials
\begin{align}
\label{rising_fact} F_n^+ (X) & := X (X+1) \dots (X+n-1) \\
\label{falling_fact} F_n^- (X) & := X (X-1) \dots (X-n+1) = (-1)^{n} F_n^+(-X).
\end{align}
By convention, $F_0^+ = F_0^- = 1$.
Theorem \ref{thm:uniform_permutation} can then be rephrased as
\begin{equation}\label{uniform_gen_fun_2}
\E \left( t^{\mathscr{C}(\si)} \right)
= \frac{1}{M!} F_{M}^+ (t),
\end{equation}
a simple identity which describes entirely the distribution of $\mathscr{C}(\sigma)$, the number of cycles of a uniform permutation. The object of interest here is the number of cycles in the commutator 
\begin{equation}
[\sigma, \tau]=\sigma \tau \sigma^{-1} \tau^{-1}
\end{equation}
between a uniform permutation $\sigma$ and a \textit{fixed} permutation $\tau$, so that the distribution at stake depends on $\tau$; more precisely, it only depends on the cycle structure of $\tau$. Diaconis \& al. \cite{Diaconis2014} studied the number of fixed points in such a commutator. Note that there are also several studies of the cycles of commutators between two uniform permutations, and even general words involving i.i.d. uniform permutations and their inverses: see for example \cites{Nic94, LP10}. Fixing $\tau$ can be thought of as a `quenched' version of these questions. \medskip

For some specific choices of $\tau$, we are able to describe exactly the distribution at stake; a related question of great interest is the asymptotic behavior of $\mathscr{C}([\sigma,\tau])$. For this, the number of fixed points of $\tau$ seems to be the main key parameter; one could hope to establish that, under the condition that $\tau$ has `few' fixed points, the number of cycles of $[\sigma,\tau]$ is asymptotically well approximated by that of a uniform even permutation. This, however, would have to be done by entirely different methods. The present work is solely concerned with a few explicit descriptions for fixed $M$, that can be derived \textit{via} random matrix theory. \\

In the rest of this section, we state our main results and establish a few preliminary facts. All proofs are provided in Section \ref{sec:rmt}.

\subsection{Results}

We prove the following results, where $F_M^{+}$ and $F_M^{-}$ denote the rising and falling factorial respectively, as defined in \eqref{rising_fact} and \eqref{falling_fact}.

\begin{theorem}\label{thm:one_cycle}
For any $M \geq 1$, if $\tau=(1, \dots, M) \in \mathfrak{S}_M$ is an $M$-cycle and $\si$ is uniformly distributed in $\mathfrak{S}_M$, then
\be\label{one_cycle_gen_fun}
\E \left( t^{\mathscr{C}([\si,\tau])} \right)
= \frac{1}{(M+1)!} \left( F_{M+1}^+(t) - F_{M+1}^-(t) \right),
\ee
which is equivalent to
\be\label{odd_interpretation}
\mathscr{C}([\si,\tau]) 
\disteq
\mathscr{C}(\nu),
\ee
where $\nu$ is distributed according to the uniform distribution on $\mathscr{A}_{M+1}^c$, the complement of the alternating group.
\end{theorem}

Theorem \ref{thm:one_cycle} is equivalent to a closed formula for Hultman numbers, which was already known from different techniques \cite{Stanley}; our input in that case is simply to provide a new conceptual proof of this formula, relying on random matrix theory. Another case that this proof technique allows to solve explicitly is when $\tau$ the product of two cycles of equal size $M$ in $\mathfrak{S}_{2M}$.

\begin{theorem}\label{thm:two_cycles}
For any $M \geq 1$, if $\tau=(1, \dots, M)(M+1, \dots, 2M) \in \mathfrak{S}_{2M}$ is a product of two disjoint $M$-cycles and $\si$ is uniformly distributed in $\mathfrak{S}_{2M}$, then
\begin{multline}
\E \left( t^{\mathscr{C}([\si,\tau])} \right)
= \frac{F_{2M+1}^+ (t) - F_{2M+1}^- (t)}{(2M+1)!} 
+ \frac{2}{(2M)!}  \left( \frac{ F_{M+1}^+ (t) -  F_{M+1}^- (t)}{M+1} \right)^2 \\
- \frac{2}{(2M)!}  \sum_{k=0}^{M} (-1)^k 
\frac{k!}{2M-k+1}
\binom{M}{k}^2 \left( F_{2M-k+1}^+ (t) + F_{2M-k+1}^+ (-t) \right)
\end{multline}
\end{theorem}

The last solvable case we present is in a sense the opposite extreme, when $\tau$ is a product of many transpositions. 

\begin{theorem}\label{thm:many_trans} For any $M\geq 1$, if $
\tau = (1,2) \cdots (2M-1, 2M) \in \mathfrak{S}_{2M}
$
is a product of $M$ disjoint transpositions, and $\sigma$ be uniformly distributed in $\mathfrak{S}_{2M}$, then
\begin{equation}
\E \left( t^{\mathscr{C}([\sigma, \tau])} \right)
=
\frac{ 2^M M!}{(2M) !} \ F_{M}^+ \left( t^2/2 \right),
\end{equation}
which is equivalent to
\be
\mathscr{C}([\si, \tau]) \disteq 2 \sum_{k=1}^M B_k
\ee
where $(B_k)_{k \geq 1}$ is a sequence of independent Bernoulli variables with $B_k \disteq \mathrm{Ber}\left( \frac{1}{2k-1}\right)$.
\end{theorem}

The present methods and results suggest possible generalization to other particular cases of $\tau$. For instance, it is in principle possible to derive the distribution of $\mathscr{C}([\sigma,\tau])$ where $\tau$ is a product of disjoints $m$-cycles, thanks to the identities of \cite{DubachPowers}. 

\subsection{First elementary facts}

Note that the commutator $[\si, \tau]$ can be interpreted as the product $\tau_1 \tau_2$, where
\begin{equation}
\tau_1 = \sigma \tau \sigma^{-1} , \qquad
\tau_2 = \tau^{-1},
\end{equation}
which implies that $\tau_1$ is a random permutation chosen uniformly among those who have the same cycle structure as $\tau$ (but independent of any other aspect of $\tau$). This is the reason why the distribution of $\mathscr{C}([\si, \tau])$ only depends on the cycle structure of $\tau$, and also explains the link with Hultman numbers and the result of Stanley \cite{Stanley}. \medskip

Another important remark is that the number of cycles in a commutator $[\si, \tau]$ always is of the same parity as $M$, the size of the permutations. This is for the simple reason that a commutator is an element of the alternating group $\mathfrak{A}_M$, which is equivalent to this property:
\begin{fact}\label{odd_fact} In the symmetric group $\mathfrak{S}_M$,
$\si \in \mathfrak{A}_M$ if and only if $\mathscr{C}(\si) \equiv M [2]$.
\end{fact}
\begin{proof}
If $\sigma \in \mathfrak{S}_M$ has cycles of sizes $C_1, \dots, C_{\mathscr{C}(\sigma)}$, then
\be
\epsilon (\sigma) = \prod_{i=1}^{\mathscr{C}(\sigma)} (-1)^{C_i+1} = (-1)^{M + \mathscr{C}(\sigma)}
\ee
and the result follows directly.
\end{proof}
This remark allows us to compute the generating function of the number of cycles in a uniform element of $\mathfrak{A}_M$.

\begin{proposition}\label{odd_distr}
If $\nu$ is uniform in $\mathfrak{A}_M$, then
\be
\E_{\mathfrak{A}_M} \left( t^{\mathscr{C}(\nu)} \right)
= \frac{1}{M!} \left( F_{M}^+(t) + F_{M}^-(t) \right),
\ee
and if $\nu$ is uniform in $  \mathfrak{A}_M^c = \mathfrak{S}_M \backslash \mathfrak{A}_M$, then
\be
\E_{\mathfrak{A}_M^c} \left( t^{\mathscr{C}(\nu)} \right)
= \frac{1}{M!} \left( F_{M}^+(t) - F_{M}^-(t) \right).
\ee
\end{proposition}

\begin{proof}
These formulas follow from Theorem \ref{thm:uniform_permutation} and a conditioning argument coherent with Fact \ref{odd_fact}: a uniform even (resp. odd) permutation is a uniform permutation, conditionally on the number of cycle having the right parity.
\end{proof}

\begin{remark} The family $(F_n^+)_{n \geq 0}$ is an eigenbasis for the discrete differential operator
\be\label{def:discrete_differential}
\mathscr{D}: P(X) \mapsto P(X) - P(X-1).
\ee
Indeed, $\mathscr{D} F_n^+ = n F_{n-1}^+$. In particular, the integer values of $F_{n}^+$ can be written in the following form, for $n,k \geq 1$:
\begin{equation}\label{fact_identity}
F_{n}^+(k) = \sum_{j=1}^k \mathscr{D} F_{n}^+(j)
 = \sum_{j=1}^k n F_{n-1}^+ (j) = n \sum_{j=1}^k  \frac{\Gamma(n+j-1)}{\Gamma(j)}.
\end{equation}
\end{remark}

\section{Random matrices}\label{sec:rmt}

This section presents a few key facts about the relevant ensembles of non-Hermitian random matrices; we then apply these to provide the proofs of the results announced in Section \ref{sec:rp}. We consider the following two types of Gaussian random matrices, related to the complex and the real Ginibre ensembles; the difference being that, for the sake of clarity, we work with \textit{unscaled} matrices, meaning that individual entries have variance $1$ instead of the usual $N^{-1/2}$. Throughout this section, then,
\begin{enumerate}
    \item $G$ is an $N \times N$ complex Gaussian matrix, by which we mean that the entries $G_{ij}$ are i.i.d. standard complex Gaussian random variable.
\item $R$ is an $N \times N$ real Gaussian matrix, by which we mean that the entries $R_{ij}$ are i.i.d. standard real Gaussian random variable. 
\end{enumerate}

The fact that all entries of these matrices are Gaussian makes it possible to access some relevant statistics using Wick's law, leading to a topological understanding of the contributions -- an approach known as \textit{genus expansion}. Some consequences of genus expansion for non-Hermitian matrices have been exposed in \cite{DubachPeled}, where a new conceptual proof of Theorem \ref{thm:uniform_permutation} was mentioned as a byproduct (p.9). Although this proof may seem a bit far-fetched (considering that the same propety follows from Feller coupling, which is a more `natural' explanation), we recall it here as it serves as a first example of the use of random matrices to answer questions about random permutations.

\begin{proof}[Proof of Theorem \ref{thm:uniform_permutation} via random matrix theory.]
 For $G$ a complex Gaussian matrix, we have, using the independence of the entries and the additivity property of independent gamma variables,
\begin{equation}
 \tr \left( G^* G \right) 
 = \sum_{i,j=1}^N |G_{ij}|^2 
 \disteq  \sum_{i,j=1}^N \gamma_1
 \disteq \gamma_{N^2},
\end{equation}
 so that the moments are given explicitly by
 $$
\E \left( \tr \left( G^* G \right)^M \right) 
= \E \left( \gamma_{N^2}^M \right)
= \frac{\Gamma(N^2 + M)}{\Gamma(N^2)}
= F_{M}^+ (N^2).
 $$
 On the other hand, we have by genus expansion,
 $$
 \E \left( \tr \left( G^* G \right)^M \right) 
 =
 \sum_{\phi \in \AdPair( G^*G, \dots, G^*G )} N^{2c(S_{\phi}) - 2g (S\phi)},
 $$
using the notations of \cite{DubachPeled}, namely: $\phi$ is an admissible pairing of edges, $S_{\phi}$ the resulting surface, $c(S_{\phi})$ its number of connected components, and $g(S_{\phi})$ the sum of their genera. In that case, as illustrated on Figure \ref{fig:Example1}, each surface $S_\phi$ obtained by the gluing of a certain numbers of $2$-gons is simply a union of spheres, so that $g (S_{\phi}) =0$. Moreover, there is a straightforward one-to-one correspondence between admissible pairings $\phi$ and elements $\si \in \mathfrak{S}_M$: the $G$-edge of face $j$ is matched with the $G^*$-edge of face $\sigma(j)$. In particular, $c (S_{\phi})= \mathscr{C}(\sigma)$. This allows us to write
 $$
 \E_{\sigma} \left( N^{2 \mathscr{C}(\sigma)} \right) 
 = \frac{1}{M!} \sum_{\sigma \in \mathfrak{S}_M} N^{2 \mathscr{C}(\sigma)}
 = \E_G \left( \tr \left( G^* G \right)^M \right) 
 = \frac{1}{M!} F_{M}^+ (N^2).
 $$
The above equality holds for any $N, M \geq 1$, so that it characterizes all relevant generating functions and proves the claim.
\end{proof}

\begin{figure}[ht]
\begin{center}
\begin{tikzpicture}
\draw (-5.3,1.7) node {$\sigma=\rm{Id} = (1)(2)(3)$};
\draw (-5.3,0) node {
\begin{tikzpicture}[scale=1.1]
\begin{scope}[thick,decoration={
    markings,
    mark=at position 0.51 with {\arrow{>}}}
    ]
\draw[thick,color=blue, postaction={decorate}] (-1,0.2) to[bend right=50] (-.8,0.7) to[bend right=50] (-1.2,0.7) to[bend right=50] (-1,0.2) ;
\draw[thick,color=blue, postaction={decorate}] (1,0.2) to[bend right=50] (1.2,0.7) to[bend right=50] (.8,0.7) to[bend right=50] (1,0.2) ;
\draw[thick,color=blue, postaction={decorate}] (0,-.8) to[bend left=50] (-.2,-.3) to[bend left=50] (.2,-.3) to[bend left=50] (0,-.8) ;
\filldraw[fill=black!5!white,draw=white] (-1,-0.5) to[bend right=50] (-1,0.5) to[bend right=50] (-1,-0.5);
\draw[postaction={decorate}] (-1,-0.5) to[bend right=50] (-1,0.5);
\draw[postaction={decorate}] (-1,-0.5) to[bend left=50] (-1,.5);
\filldraw[fill=black!5!white,draw=white] (1,-0.5) to[bend right=50] (1,0.5) to[bend right=50] (1,-0.5);
\draw[postaction={decorate}] (1,-0.5) to[bend right=50] (1,0.5);
\draw[postaction={decorate}] (1,-0.5) to[bend left=50] (1,.5);
\filldraw[fill=black!5!white,draw=white] (0,-1.5) to[bend right=50] (0,-.5) to[bend right=50] (0,-1.5);
\draw[postaction={decorate}] (0,-.5) to[bend left=50] (0,-1.5);
\draw[postaction={decorate}] (0,-.5) to[bend right=50] (0,-1.5);
\draw (-.5,0) node {$G$};
\draw (-1.5,0) node {$G^*$};
\draw (1.5,0) node {$G$};
\draw (.5,0) node {$G^*$};
\draw (.5,-1) node {$G^*$};
\draw (-.5,-1) node {$G$};
\draw (-1,0) node {$1$};
\draw (1,0) node {$2$};
\draw (0,-1) node {$3$};
\end{scope}
\end{tikzpicture}
}; 
\draw (0,1.7) node {$\sigma=(1,2)(3)$};
\draw (0,-.2) node {
\begin{tikzpicture}[scale=1.1]
\begin{scope}[thick,decoration={
    markings,
    mark=at position 0.51 with {\arrow{>}}}
    ] 
\draw[thick,color=blue,postaction={decorate}] (-1,0) to[bend right=40] (1,0);
\draw[thick,color=blue,postaction={decorate}] (1,.2) to[bend right=90] (1.1,.7) to[bend right=20] (-1.1,.7) to[bend right=90] (-1,.2);
\draw[thick,color=blue,postaction={decorate}] (0,-1.2) to[bend right=50] (-.2,-1.7) to[bend right=50] (.2,-1.7) to[bend right=50] (0,-1.2) ;
\filldraw[fill=black!5!white,draw=white] (-1,-0.5) to[bend left=45] (-1,0.5) to[bend left=45] (-1,-0.5);
\draw[postaction={decorate}] (-1,-0.5) to[bend right=50] (-1,0.5);
\draw[postaction={decorate}] (-1,-0.5) to[bend left=50] (-1,.5);
\filldraw[fill=black!5!white,draw=white] (1,-0.5) to[bend right=50] (1,0.5) to[bend right=50] (1,-0.5);
\draw[postaction={decorate}] (1,-0.5) to[bend right=50] (1,0.5);
\draw[postaction={decorate}] (1,-0.5) to[bend left=50] (1,.5);
\filldraw[fill=black!5!white,draw=white] (0,-1.5) to[bend right=50] (0,-.5) to[bend right=50] (0,-1.5);
\draw[postaction={decorate}] (0,-.5) to[bend left=50] (0,-1.5);
\draw[postaction={decorate}] (0,-.5) to[bend right=50] (0,-1.5);
\draw (-.5,0) node {$G$};
\draw (-1.5,0) node {$G^*$};
\draw (1.5,0) node {$G$};
\draw (.5,0) node {$G^*$};
\draw (.5,-1) node {$G^*$};
\draw (-.5,-1) node {$G$};
\draw (-1,0) node {$1$};
\draw (1,0) node {$2$};
\draw (0,-1) node {$3$};
\end{scope}
\end{tikzpicture}
};
\draw (5.3,1.7) node {$\sigma=(1,2,3)$};
\draw (5.3,-.1) node {
\begin{tikzpicture}[scale=1.1]
\begin{scope}[thick,decoration={
    markings,
    mark=at position 0.51 with {\arrow{>}}}
    ] 
\draw[thick,color=blue,postaction={decorate}] (-1,.1) to[bend left=20] (1,.1);
\draw[thick,color=blue,postaction={decorate}] (1,0) to[bend left=60] (1.5,-.5) to[bend left=40] (0,-1.3);
\draw[thick,color=blue,postaction={decorate}] (0,-1.3) to[bend left=40] (-1.5,-.5) to[bend left=60] (-1,0);
\filldraw[fill=black!5!white,draw=white] (-1,-0.5) to[bend right=50] (-1,0.5) to[bend right=50] (-1,-0.5);
\draw[postaction={decorate}] (-1,-0.5) to[bend right=50] (-1,0.5);
\draw[postaction={decorate}] (-1,-0.5) to[bend left=50] (-1,.5);
\filldraw[fill=black!5!white,draw=white] (1,-0.5) to[bend right=45] (1,0.5) to[bend right=50] (1,-0.5);
\draw[postaction={decorate}] (1,-0.5) to[bend right=50] (1,0.5);
\draw[postaction={decorate}] (1,-0.5) to[bend left=50] (1,.5);
\filldraw[fill=black!5!white,draw=white] (0,-1.5) to[bend right=50] (0,-.5) to[bend right=50] (0,-1.5);
\draw[postaction={decorate}] (0,-.5) to[bend left=50] (0,-1.5);
\draw[postaction={decorate}] (0,-.5) to[bend right=50] (0,-1.5);
\draw (-.5,0) node {$G$};
\draw (-1.5,0) node {$G^*$};
\draw (1.5,0) node {$G$};
\draw (.5,0) node {$G^*$};
\draw (.5,-1) node {$G^*$};
\draw (-.5,-1) node {$G$};
\draw (-1,0) node {$1$};
\draw (1,0) node {$2$};
\draw (0,-1) node {$3$};
\end{scope}
\end{tikzpicture}
};
\end{tikzpicture}
\end{center}
\vspace{-.2in}
\caption{Illustration of the random matrix proof of Theorem \ref{thm:uniform_permutation} for $M=3$. We represent three admissible pairings (out of a total of $6$) corresponding to the identity, a transposition and a $3$-cycle; the resulting surfaces are unions of spheres: three, two and one respectively.}
\label{fig:Example1}
\end{figure}
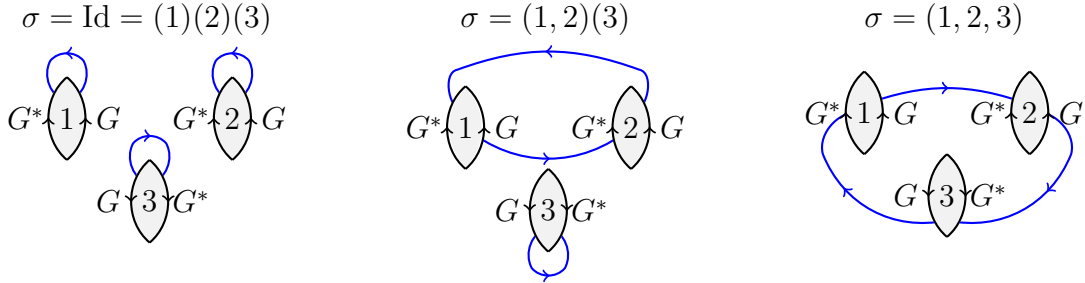

The proof of Theorem \ref{thm:many_trans} relies on the same phenomenon, with a real Gaussian matrix $R$ (see Section \ref{sec:last}). The proof of Theorems \ref{thm:one_cycle} and \ref{thm:two_cycles} on the other hand rely on another feature, specific to the complex Gaussian matrix $G$, namely: the decorrelation of high powers of its eigenvalues. This fact is known to hold more generally for determinantal point processes with radial symmetry in the complex plane.

\begin{theorem}[Hough, Krishnapur, Peres, Vir\'ag]\label{thm:high_powers}
Let $\{ \la_1, \cdots, \la_N \}$ be the set of eigenvalues of $G$ defined as above. For any $M \geq N$, the following identity in distribution holds
\begin{equation}
\{ \la_1^M, \cdots, \la_N^M \} \disteq \{ \gamma_1^M e^{i \theta_1}, \dots, \gamma_N^M e^{i \theta_N} \}
\end{equation}
where the $\theta_i$'s are i.i.d. \! real variables uniform in $[0,2\pi]$, and $\gamma_1, \dots, \gamma_N$ are independent gamma variables with corresponding parameters $1, \dots, N$.
\end{theorem}
See \cite{HKPV} for a general presentation of this rather counter-intuitive property, and \cite{DubachPowers} for a generalization to small powers $M < N$. \medskip

The following lemma is the essential link between complex Gaussian random matrices and random permutations that will be used repeatedly in this note.

\begin{lemma}\label{lem:corresp}
If $C_1 + \cdots+ C_K =M$, and $G$ is a complex Gaussian matrix as defined above, then
\begin{equation}
\E_{G} \left( \prod_{k=1}^K \left| \tr G^{C_k} \right|^2 \right)
=
M! \
\E_{\sigma} \left( N^{\mathscr{C}([\sigma, \tau])} \right),
\end{equation}
where $\tau$ has $K$ cycles of respective sizes $C_1, \dots C_K$, and $\sigma \in \mathfrak{S}_M$ is a uniform permutation.
\end{lemma}

\begin{proof}
This follows from the diagramatic approach for non-Hermitian matrices presented in \cite{DubachPeled}, and the following remarks. First, we choose an arbitrary numbering of the $G$-edges from $1$ to $M$, and also number the $G^*$-edges symmetrically. To each admissible pairing (i.e. a bipartite pairing between the $G$-edges and the $G^*$-edges), one associates the permutation $\sigma \in \mathfrak{S}_M$ such that $\sigma(i)=j$ if and only if the $G$-edge number $i$ is paired with the $G^*$-edge number $j$. We call $\tau$ the fixed permutation such that $\tau(i)=j$ if and only if the $G$-edge $i$ follows directly the $G$-edge $j$ in the diagram, which implies that $\tau$ has the corresponding cycle structure. The key remark is then that
\begin{equation}
V(S_{\phi}) = \mathscr{C}([\sigma, \tau]),
\end{equation}
which can be justified in the following way: there are initially $M$ vertices of the $G$-edges, which eventually will be associated to the symmetric $M$ vertices of the $G^*$-edges; $V(S_{\phi})$ being the number of vertices remaining after the associations. It is enough to count the cycles of associations created in the $G$-vertices. If one defines `vertex number $i$' as the end-vertex of the oriented edge number $i$, the pairing associated to $\sigma$ forces the assimilation of the $G$-vertex $i$ with the $G^*$-vertex number $\tau (\sigma(i))$, and consequently the assimilation of the $G$-vertices encoded in permutation $\tau^{-1} \sigma^{-1} \tau \sigma$. Therefore, the number of free vertices in the end is the number of cycles in the commutator $[\sigma, \tau]$.
\end{proof}

\begin{remark}
We can also write, using Euler's identity:
\begin{equation}
\mathscr{C}([\sigma,\tau]) 
= V(S_{\phi})
= \chi(S_{\phi}) + E(S_{\phi}) -F(S_{\phi}) = M - 2(g(S_{\phi}) + K - c(S_{\phi}))
\end{equation}
It follows, in particular, that $M- \mathscr{C}([\si,\tau])$ is even, coherently with Fact \ref{odd_fact}.
\end{remark}

A closely related random matrix interpretation has been previously suggested by Alexeev \& Zograf \cites{Alexeev2011, Alexeev2014}, who related Hultman numbers to the quantities $\E \tr ( G^M G^{*M} )$, and therefore to the singular values of $G^M$. Lemma \ref{lem:corresp}, on the other hand, relates Hultman numbers to $\E |\tr G^M|^2$, and more generally, statistics of cycles of commutators to eigenvalues of powers of $G$. It is this choice of statistics that enables explicit computations, for instance \textit{via} Theorem \ref{thm:high_powers}.

\subsection{Proof of Theorem \ref{thm:one_cycle}: one large cycle}

\begin{proof}[Proof of Theorem \ref{thm:one_cycle}]
Let $M \geq 1$ be a fixed integer. We denote
$
P(t) := \E_{\sigma} \left( t^{\mathscr{C}([\si,\tau])} \right)
$
the generating function of the number of cycles of a commutator between a uniform permutation $\sigma$ and a large cycle $\tau$ in $\mathfrak{S}_M$. Clearly, $P$ is a polynomial of degree $M$, and $P(0)=0$. For any $N\geq 1$, let $G$ be a complex Gaussian matrix of size $N \times N$. Lemma \ref{lem:corresp} gives
\begin{equation}
 P(N) = \frac{1}{M!} \E_G | \tr G^M |^2.
\end{equation}
We compute this quantity using Theorem \ref{thm:high_powers}, for any $1 \leq N \leq M$:
\begin{equation}
\E | \tr G^M |^2 = \sum_{i=1}^N \E |\la_i|^{2M} = \sum_{i=1}^N \frac{\Ga(M+i)}{\Ga(i)},
\end{equation}
so that $P(0) = 0$ and for every $1 \leq N \leq M$, using (\ref{fact_identity}),
\begin{equation}
P (N) = \frac{1}{M!} \sum_{i=1}^N \frac{\Ga(M+i)}{\Ga (i)} = \frac{1}{(M+1)!} F_{M+1}^+(N).
\end{equation}
Thus, the polynomial 
$$
F_{M+1}^+(t) - (M+1)! P(t)
$$
is of degree $M+1$ and cancels at $t=0, 1, \dots, M$. Therefore it is a multiple of $F_{M+1}^-$. By inspection of the leading coefficient, it is exactly $F_{M+1}^-$. The main result follows; consequence \eqref{odd_interpretation} holds by Proposition \ref{odd_distr}.
\end{proof}

One remarquable consequence of this identity is the fact that $\mathscr{C}([\sigma,\tau])$ is distributed like a sum of independent Bernoulli variables, although their parameters are not as straightforward as those appearing in Theorems \ref{thm:uniform_permutation} and \ref{thm:many_trans}.

\begin{corollary}
If $\tau = (1, \dots, M) \in \mathfrak{S}_M$, there are parameters $(p_j)_{j=1}^{\lfloor M/2 \rfloor}$ such that
$$
M - \mathscr{C}([\si,\tau]) \disteq \sum_{j=1}^{\lfloor M-1/2 \rfloor } 2 B_{p_j}
$$
where the $B_p$'s are independent Bernoulli variables.
\end{corollary}

\begin{proof}
The generating function (\ref{one_cycle_gen_fun}) satisfies the Lee-Yang property, i.e. it has only purely imaginary zeros. Indeed, if $\Re z > 0$, then for every $j=1, \dots, M$, $|z-j|>|z+j|$. But if $z \neq 0$ is a zero of $P_M = F_{M+1}^+ - F_{M+1}^-$, then obviously $z \neq 1, \dots, M$ and
$$
1 = \left| \frac{F_M^+ (z)}{F_M^- (z)} \right| = \prod_{j=1}^M \frac{|z+j|}{|z-j|} < 1
$$
which is a contradiction. The conclusion follows by using the parity of $P_M$ (same as $M$). It follows that the decomposition of $P_M$ as a product of irreducible polynomials in $\R[X]$ is
$$
P_M(X) = X^{1 + \delta_{2 | M}} \prod_{j=1}^{\lfloor M/2 \rfloor} (X^2 + |z_j|^2)
$$
where $z_j$ are the (purely imaginary) roots. Each factor can then be interpreted as the generating function of a variable $2 B_{p_j}$, with $p_j = \frac{1}{1+|z_j|^2}$.
\end{proof}

 \subsection{Proof of Theorem \ref{thm:two_cycles}: two large cycles}

For any $M$, we define the polynomial of degree $2M+1$
\be \label{def:Q}
Q_{2M+1} (X) := \sum_{k=0}^{M} (-1)^k 
\frac{k!}{2M-k+1}
\binom{M}{k}^2 F_{2M-k+1}^+ (X)
\ee
The proof of Theorem \ref{thm:two_cycles} requires to identify a particular values of this polynomial, which is done in Lemma \ref{lem:identify_Q}.

\begin{lemma}\label{lem:identify_Q}
The following identity holds:
\be\label{ideux}
\sum_{k=1}^N \left( \frac{\Gamma(k+M)}{\Gamma(k)} \right)^2
= Q_{2M+1}(N).
\ee
\end{lemma}

\begin{proof}
We will make use of the following classical identity, for any $m \leq n$,
\be\label{connection_coefs}
F_m^- (X) F_n^- (X) 
= \sum_{k=0}^{m} \binom{m}{k} \binom{n}{k} k! \ F_{m+n-k}^- (X),
\ee 
which can be combinatorially understood, for integer values of $X$, as counting in two different ways the number of possible pairs of lists of length $m$ and $n$ respectively, chosen among $X$ possible items. It follows from identity \eqref{connection_coefs} with $m=n=M+1$ that
\be
F_{M+1}^+ (X)^2
= \sum_{k=0}^{M+1} (-1)^k \binom{M+1}{k}^2 k! F_{2M+2-k}^+ (X) 
\ee
So that $Q_{2M+1}$ is the unique polynomial of degree $2M+1$ such that
\be
Q_{2M+1}(0) = 0 \quad \& \quad
\mathscr{D} Q_{2M+1} = (F_{M}^+)^2.
\ee
In particular, for every $N \geq 1$,
\be
Q_{2M+1}(N) 
= \sum_{k=1}^N \mathscr{D} Q_{2M+1}(k)
= \sum_{k=1}^N F_M^+(k)^2
\ee
which is the claim.
\end{proof}

We can now establish Theorem \ref{thm:two_cycles}, following the same method as for Theorem \ref{thm:one_cycle}.

\begin{proof}[Proof of Theorem \ref{thm:two_cycles}]
Let $M\geq 1$ be a fixed integer and denote 
$P(t):= \E_{\sigma} (t^{\mathscr{C}([\sigma, \tau])})$, where $\tau = (1, \dots, M)(M+1, \dots, 2M) \in \mathfrak{S}_{2M}$ and $\sigma \in \mathfrak{S}_{2M}$ is uniform. $P$ is an polynomial of degree $2M$ such that $P(0)=0$. By Fact \ref{odd_fact}, $P$ is also even. For any $N\geq 1$, Lemma \ref{lem:corresp} gives
\begin{equation}
P(N) 
= \frac{1}{(2M)!} \E |\tr G^M |^4.
\end{equation}
This quantity can be computed exactly for $1 \leq N \leq M$, using Theorem \ref{thm:high_powers} and then Lemma \ref{lem:identify_Q}:
\begin{align*}
\E |\tr G^M |^4 & = \E \left( \sum_{i,j,k,l=1}^N \la_i^M \la_j^M \overline{\la_k}^M \overline{\la_l}^M \right) \\
& = \sum_{i=1}^N \E |\la_i|^{4M} + 2\left(\sum_{i=1}^N \E |\la_i|^{2M}\right)^2  - 2\sum_{i=1}^N \left(\E |\la_i|^{2M}\right)^2 \\
& = \sum_{i=1}^N \frac{\Gamma(i+2M)}{\Gamma(i)} + 2 \left( \sum_{i=1}^N \frac{\Gamma(i+M)}{\Gamma(i)} \right)^2 - 2 \sum_{i=1}^N \left( \frac{\Gamma(i+M)}{\Gamma(i)} \right)^2 \\
& = \frac{1}{2M+1} F_{2M+1}^+(N) + \frac{2}{(M+1)^2} F_{M+1}(N)^2 - 2 Q_{2M+1}(N).
\end{align*}
so that we have, for $ 0 \leq N \leq M$,
\begin{align*}
P(N) & = \frac{1}{(2M+1)!} F_{2M+1}^+ (N) + \frac{2}{(2M)!(M+1)^2}  F_{M+1}^+ (N)^2 - \frac{2}{(2M)!} Q_{2M+1} (N).
\end{align*}
Note that all terms $F_{2M+1}^+, F_{M+1}^+, Q_{2M}$ cancel at $-M, \dots, 0$. Therefore, the following polynomials have degree $2M$ and, by parity, coincide with $P$ on $2M+1$ values:
\begin{multline}
\frac{F_{2M+1}^+ (t) - F_{2M+1}^- (t)}{(2M+1)!}
+ \frac{2}{(2M)!}  \left( \frac{ F_{M+1}^+ (t) - F_{M+1}^- (t)}{M+1} \right)^2 \\
- \frac{2}{(2M)!} \left( Q_{2M+1} (t) + Q_{2M+1} (-t) \right)
\end{multline}
therefore it is equal to $P$, which is the claim.
\end{proof}
%

\subsection{Proof of theorem \ref{thm:many_trans}: product of many transpositions}\label{sec:last}

We conclude this note with a proof of Theorem \ref{thm:many_trans}, as well as a byproduct of this method: namely, identities in distribution for $\tr (G^2)$ and $\tr (G_1 G_2)$ where $G_1,G_2$ are i.i.d. complex Gaussian matrices.

\begin{proof}[Proof of Theorem \ref{thm:many_trans}]
Let $R$ be a real Gaussian matrix as defined in the beginning of Section \ref{sec:rmt}. In particular, its coefficients are i.i.d. variables such that $R_{ij}^2 \disteq 2 \gamma_{1/2}$, so that we have, on the one hand,
\begin{equation}
\tr ( R R^t ) = \sum_{i,j=1}^N R_{ij}^2
\disteq
\sum_{i,j=1}^N 2 \ga_{1/2}
\disteq 2 \ga_{N^2/2},
\end{equation}
which moments are given by
\begin{equation}
\E (\tr R R^t )^M = 2^M \frac{\Gamma(N^2/2 + M)}{\Gamma(N^2/2)}
= 2^M F_{M}^+ (N^2/2).
\end{equation}
On the other hand, in the genus expansion for real Gaussian matrices presented in \cite[Section 4.4]{DubachPeled}, the admissible pairings include all associations of $R$-edges and $R^t$-edges, with the \textit{caveat} that the corresponding edge-gluing is either direct (for mixed pairs $R$--$R^t$) or indirect (for pairs $R$--$R$ or $R^t$--$R^t$). In general, the resulting surface need not be oriented, and connected components may include, for instance, real projective planes, Klein bottles, etc. However, in this case, as in the proof of Theorem \ref{thm:uniform_permutation}, one can see that the resulting surface is always a union of spheres: indeed, only $2$-gons are being used, so that the resulting connected components are either spheres or projective planes; moreover, no connected component is a projective plane, as indirect pairings always happens an even number of times due to the specific structure of the faces (all have one $R$ edge and one $R^t$ edge). \medskip

The next important remark is that an admissible pairing between these $2M$ edges can be interpreted as a permutation $\nu \in \mathfrak{S}_{2M}$ which is a product of $M$ disjoint transpositions: one numbers the edges from $1$ to $2M$ and pairs edges $i$ and $j$ whenever $\nu(i)=j$, $\nu(j)=i$. Note that the main difference with the proof of Theorem \ref{thm:uniform_permutation} is that the pairing is not bipartite here, as all associations are admissible. Moreover, in the resulting surface, 
\begin{equation}
V(S_{\phi}) = 2 c(S_{\phi}) - 2 g(S_{\phi}) = 2 c(S_{\phi}) = \mathscr{C}(\nu \tau)
\end{equation}
where $\tau$ is another product of $2M$ disjoint transpositions, which corresponds to the edges themselves (so that $\tau = (1,2) \cdots (2M-1,2M)$ if the edges have been numbered from left to right). If we denote by $\mathfrak{D}_{2M}$ the set of such products of transpositions, which has cardinal $(2M-1)!!$, we find that
$$
\E (\tr R R^t )^M 
=
\sum_{\phi \in \AdPair(RR^t, \dots, RR^t)} N^{ V(\Gamma_{\phi})}
=
\sum_{\nu \in \mathfrak{D}_{2M}} N^{ \mathscr{C}(\nu \tau)}
=
(2M-1)!! \ \E_{\sigma} \left( N^{ \mathscr{C}( [ \si, \tau ] )} \right)
$$
where $\nu \disteq \sigma \tau \sigma^{-1}$ with $\sigma$ uniform. We conclude that:
$$
\E_{\sigma} \left( N^{\mathscr{C}([ \si, \tau ] )} \right) =
\frac{2^M}{(2M-1)!!}
F_{M}^+(N^2/2)
= 
\prod_{k=1}^{M} \frac{N^2 + 2 k -2 }{2k-1}.
$$
As this is true for any $N,M \geq 1$, we deduce that for any $M$, the relevant generating function is given by
$$
\E_{\sigma} \left( t^{ \mathscr{C}(\si, \tau)} \right)
=
\prod_{k=1}^{M} \frac{t^2 + 2 k-2 }{2k-1}
=
\prod_{k=1}^{M} \E \left( t^{2 \mathrm{Ber}\left(\frac{1}{2k-1}\right))} \right),
$$
which is the claim.
\end{proof}

\begin{figure}[ht]
\begin{center}
\begin{tikzpicture}
\draw (-5.3,1.2) node {$\nu=(1,2)(3,4)$};
\draw (-5.3,0) node {
\begin{tikzpicture}
\begin{scope}[thick,decoration={
    markings,
    mark=at position 0.52 with {\arrow{>}}}
    ]
\draw[thick,color=blue] (-1,0.2) to[bend right=50] (-.8,0.7) to[bend right=50] (-1.2,0.7) to[bend right=50] (-1,0.2) ;
\draw[thick,color=blue] (1,0.2) to[bend left=50] (.8,0.7) to[bend left=50] (1.2,0.7) to[bend left=50] (1,0.2) ;
\filldraw[fill=black!5!white,draw=white] (-1,-0.5) to[bend right=45] (-1,0.5) to[bend right=45] (-1,-0.5);
\draw[postaction={decorate}] (-1,-0.5) to[bend right=45] (-1,0.5);
\draw[postaction={decorate}] (-1,-0.5) to[bend left=45] (-1,.5);
\filldraw[fill=black!5!white,draw=white] (1,-0.5) to[bend right=45] (1,0.5) to[bend right=45] (1,-0.5);
\draw[postaction={decorate}] (1,-0.5) to[bend right=45] (1,0.5);
\draw[postaction={decorate}] (1,-0.5) to[bend left=45] (1,.5);
\draw (-.5,0) node {$R$};
\draw (-1.5,0) node {$R^t$};
\draw (1.5,0) node {$R$};
\draw (.5,0) node {$R^t$};
\end{scope}
\end{tikzpicture}
}; 
\draw (0,1.2) node {$\nu=(1,4)(2,3)$};
\draw (0,-.2) node {
\begin{tikzpicture}
\begin{scope}[thick,decoration={
    markings,
    mark=at position 0.52 with {\arrow{>}}}
    ] 
\draw[thick,color=blue] (-1,0.1) to[bend left=50] (1,0.1);
\draw[thick,color=blue] (1,-.2) to[bend left=90] (1.1,-.7) to[bend left=20] (-1.1,-.7) to[bend left=90] (-1,-.2);
\filldraw[fill=black!5!white,draw=white] (-1,-0.5) to[bend right=45] (-1,0.5) to[bend right=45] (-1,-0.5);
\draw[postaction={decorate}] (-1,-0.5) to[bend right=45] (-1,0.5);
\draw[postaction={decorate}] (-1,-0.5) to[bend left=45] (-1,.5);
\filldraw[fill=black!5!white,draw=white] (1,-0.5) to[bend right=45] (1,0.5) to[bend right=45] (1,-0.5);
\draw[postaction={decorate}] (1,-0.5) to[bend right=45] (1,0.5);
\draw[postaction={decorate}] (1,-0.5) to[bend left=45] (1,.5);
\draw (-.5,0) node {$R$};
\draw (-1.5,0) node {$R^t$};
\draw (1.5,0) node {$R$};
\draw (.5,0) node {$R^t$};
\end{scope}
\end{tikzpicture}
};
\draw (5.3,1.2) node {$\nu=(1,3)(2,4)$};
\draw (5.3,-.1) node {
\begin{tikzpicture}
\begin{scope}[thick,decoration={
    markings,
    mark=at position 0.52 with {\arrow{>}}}
    ] 
\draw[thick,color=blue] (1,0.2) to[bend right=80] (1.2,.65) to[bend right=40] (-1,.15);
\draw[thick,color=blue] (-1,-0.2) to[bend right=80] (-1.2,-.65) to[bend right=40] (1,-0.15);
\filldraw[fill=black!5!white,draw=white] (-1,-0.5) to[bend right=45] (-1,0.5) to[bend right=45] (-1,-0.5);
\draw[postaction={decorate}] (-1,-0.5) to[bend right=45] (-1,0.5);
\draw[postaction={decorate}] (-1,-0.5) to[bend left=45] (-1,.5);
\filldraw[fill=black!5!white,draw=white] (1,-0.5) to[bend right=45] (1,0.5) to[bend right=45] (1,-0.5);
\draw[postaction={decorate}] (1,-0.5) to[bend right=45] (1,0.5);
\draw[postaction={decorate}] (1,-0.5) to[bend left=45] (1,.5);
\draw (-.5,0) node {$R$};
\draw (-1.5,0) node {$R^t$};
\draw (1.5,0) node {$R$};
\draw (.5,0) node {$R^t$};
\end{scope}
\end{tikzpicture}
};
\end{tikzpicture}
\end{center}
\vspace{-.1in}
\caption{Illustration of the proof of Theorem \ref{thm:many_trans} for $M=2$. The three possible pairings correspond to the three double-transpositions of $\mathfrak{S}_4$.}
\label{Example3}
\end{figure}
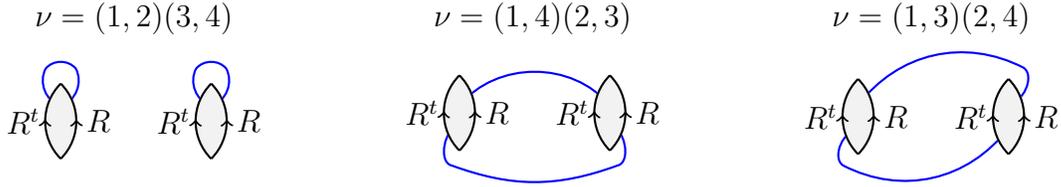

According to Lemma \ref{lem:corresp}, the same statistics can be accessed from the quantities $\E |\tr G^2 |^M$. This gives us the following description of the random variable $\tr (G^2)$.

\begin{proposition}\label{distrib_of_tr_G2}
Let $G$ be an complex Gaussian matrix, then
$$
\tr ( G^2 ) 
\disteq
e^{i \theta} \sqrt{ 4 \ga_1 \ga_{\frac{N^2}{2}}}
$$
where $\theta$ is uniform in $[0,2 \pi]$ and all variables are independent. 
\end{proposition}
\begin{proof}
We proceed according to the following steps. For every $M \geq 1$, with $\tau$ a product of $M$ disjoint transpositions,
\begin{align*}
\E_{G} \left( |\tr(G^{2})|^{2M} \right)
& =
(2M)! \ 
\E_{\sigma} \left( N^{\mathscr{C}([\si, \tau])} \right)  & \text{(Lemma \ref{lem:corresp})} \\
& =
\frac{(2M)!}{(2M-1)!!}
\E_{R} \left(
\tr \left( R R^t \right)^M \right) & \text{(Genus expansion for $R$)} \\
& = 2^M M! \ \E \left( \left( 2 \ga_{N^2/2}\right)^M \right) & \text{(Distribution of $\tr ( R R^t)$)} \\
& = \E \left( \left( 4 \ga_1 \ga_{N^2/2} \right)^M \right) &
\end{align*}
The claim follows by equality of all relevant moments, as 
$$
\E \left( \tr G^{M_1} \tr G^{*M_2} \right) = 0
$$
when $M_1 \neq M_2$.
\end{proof}

Note that Proposition \ref{distrib_of_tr_G2} with $N=1$ is a particular case of a general identity for product of gamma distributions when parameters are half of consecutive integers. \\
%
%

It makes sense to compare Proposition \ref{distrib_of_tr_G2} with the corresponding identity with two independent complex Gaussian matrices:

\begin{proposition}\label{distrib_of_tr_G1G2}
Let $G_1, G_2$ be i.i.d. complex Gaussian matrices, then
$$
\tr ( G_1 G_2 ) 
\disteq
e^{i \theta} \sqrt{ \ga_1 \ga_{N^2}}
$$
where $\theta$ is uniform in $[0,2 \pi]$ and all variables are independent. 
\end{proposition}

\begin{proof}
Genus expansion yields cycles of $\sigma=\sigma_1 \sigma_2$, which is uniform:
$$
\E |\tr G_1 G_2|^{2M}
= (M!)^2 \E(N^{\mathscr{C}(\sigma_1 \sigma_2)})
= (M!)^2
\E_{\sigma}(N^{2\mathscr{C}(\sigma)})
= M! (M+N^2)!,
$$
and the result follows.
\end{proof}

\section*{Acknowledgments}

The author acknowledges funding from the European Union's Horizon 2020 research and innovation programme, under the Marie Sk{\l}odowska-Curie Grant Agreement No. 754411. \medskip

I would like to thank Percy Deift, Paul Bourgade, Yuval Peled and Igor Kortchemski for many discussions and advice on that topic.

\end{document}